\documentclass[11pt]{amsart}
\usepackage[english]{babel}

\usepackage{amssymb}
\usepackage{amsmath,amsthm}
\usepackage{mathtools}
\usepackage{dsfont}
\usepackage{nicefrac}
\usepackage{cases}
\usepackage{bm}
\usepackage{enumitem}
\usepackage{bbm}

\usepackage[T1]{fontenc}
\usepackage{lmodern}
\usepackage{microtype}

\usepackage{empheq}

\usepackage[inner=3cm,outer=3cm,top=3.3cm,bottom=3cm]{geometry}

\usepackage{xspace}

\usepackage[autostyle]{csquotes}
\usepackage[dvipsnames]{xcolor}

\usepackage{tikz}
\usepackage{tikz-cd}
\usetikzlibrary{matrix,arrows,patterns,decorations.pathreplacing,svg.path}
\usepackage{pgfplots}
\usepackage{color}

\usepackage{todonotes}

\usepackage[colorlinks, bookmarksdepth=3, citecolor=blue,linkcolor=blue]{hyperref}

\usepackage{pdfcomment}

\makeatletter
\@ifpackageloaded{hyperref}{%
\DeclareRobustCommand{\SkipTocEntry}[5]{}}{%
\DeclareRobustCommand{\SkipTocEntry}[4]{}}
\makeatother

\newenvironment{pf}{\proof[\proofname]}{\endproof}
\theoremstyle{plain}
\newtheorem{Th}{Theorem}[section]
\newtheorem{Prop}[Th]{Proposition}

\newtheorem{Lem}[Th]{Lemma}

\numberwithin{equation}{section}
\numberwithin{figure}{section}
\theoremstyle{definition}
\newtheorem{Def}[Th]{Definition}
\newtheorem{Ex}[Th]{Example}
\newtheorem{Rem}[Th]{Remark}

\DeclareMathOperator{\conv}{conv}

\newcommand{\cal}[1]{\mathcal{#1}}
\newcommand{\C}{\mathbb C}
\newcommand{\Q}{\mathbb Q}
\newcommand{\Z}{\mathbb Z}
\newcommand{\R}{\mathbb R}

\newcommand{\N}{\mathbb N}

\newcommand{\G}{\Gamma}

\newcommand{\D}{\Delta}

\newcommand{\bK}{\mathbf{K}}

\newcommand{\cP}{\mathcal P}

\newcommand{\cV}{\cal V}

\newcommand{\la}{\langle}
\newcommand{\ra}{\rangle}

\DeclareMathOperator{\V}{V}

\DeclareMathOperator{\vv}{v}
\DeclareMathOperator{\w}{w}

\DeclareMathOperator{\Vol}{Vol}
\DeclareMathOperator{\vol}{vol}

\DeclareMathOperator{\GL}{GL}
\DeclareMathOperator{\AGL}{AGL}

\DeclareMathOperator{\Sym}{\mathbf S}

\newcommand{\im}{\operatorname{Im}}

\newcommand{\SFV}{\cV^{\,\rm{sf}}}

\newcommand{\multiset}[2]{
\left.\mathchoice
  {\left(\kern-0.48em\binom{#1}{#2}\kern-0.48em\right)}
  {\big(\kern-0.30em\binom{\smash{#1}}{\smash{#2}}\kern-0.30em\big)}
  {\left(\kern-0.30em\binom{\smash{#1}}{\smash{#2}}\kern-0.30em\right)}
  {\left(\kern-0.30em\binom{\smash{#1}}{\smash{#2}}\kern-0.30em\right)}
\right.}

\makeatletter
\renewcommand{\boxed}[2][\fboxsep]{{%
  \setlength{\fboxsep}{#1}\fbox{\m@th$\displaystyle#2$}}}
\makeatother

\newcommand{\rs}[1]{Section~\ref{S:#1}}
\newcommand{\rl}[1]{Lemma~\ref{L:#1}}
\newcommand{\rp}[1]{Proposition~\ref{P:#1}}

\newcommand{\re}[1]{(\ref{e:#1})}

\newcommand{\rt}[1] {Theorem~\ref{T:#1}}

\newcommand{\rf}[1]{Figure~\ref{F:#1}}

\newcommand{\scr}[1]{{\large\text{\fontfamily{pzc}\selectfont #1}}}
\newcommand{\Pl}{\scr{Pl}\,}

\makeatletter
\@namedef{subjclassname@2020}{%
  \textup{2020} Mathematics Subject Classification}
\makeatother

\title[Discrete Heine-Shephard problem for four lattice polygons]{
On the discrete Heine-Shephard problem\\ for four lattice polygons}
\author{Darren Gerrity}
\address{Department of Mathematics and Statistics, Cleveland State University,  2121 Euclid Ave, Cleveland, Ohio, 44115 USA}
\email{d.gerrity@vikes.csuohio.edu}
\author{Ivan~Soprunov}
\address{Department of Mathematics and Statistics, Cleveland State University,  2121 Euclid Ave, Cleveland, Ohio, 44115 USA}
\email{i.soprunov@csuohio.edu}

\begin{document}
\selectlanguage{english}

\date{}

\keywords{volume polynomial, geometric inequalities, mixed volume, lattice polytopes, lattice width, Plücker-type inequality, intersection numbers}
\subjclass[2020]{Primary 52B20, 52A40; Secondary 52A39, 14M25}

\maketitle

\begin{abstract}
We study the set of square-free parts of volume polynomials associated with four planar lattice polytopes. This is motivated by the problem of describing possible pairwise intersection numbers of four curves in $(\C^*)^2$ with prescribed Newton polytopes and generic coefficients. It is known that for arbitrary convex bodies in $\R^2$, the corresponding square-free polynomials are characterized by the Plücker-type inequalities. We show that this characterization fails in the lattice setting: the interior of the space defined by the Plücker-type inequalities contains integer polynomials that are and are not realizable by lattice polytopes. This phenomenon arises from additional arithmetic constraints on the mixed areas of lattice polytopes. These constraints become apparent when we study a ``discrete diagram'', which maps a pair of planar lattice polytopes to their mixed area together with their lattice widths in a given direction.
\end{abstract}

\section{Introduction}

In \cite{Min03} Minkowski introduced the notion of mixed volume $\vv(K_1,\dots, K_d)$ of any collection of $d$ convex bodies $K_1,\dots, K_d$ in $\R^d$. It is the unique symmetric and multiliniear (with respect to Minkowski addition) function which coincides with the volume when all the bodies are the same, $K_1=\dots=K_d$. Many geometric invariants of a convex body, such as surface area, mean width, or volume of a projection can be expressed using mixed volumes. In 1970's, thanks to the results of Bernstein, Kushnirenko, and Khovanskii \cite{Bernstein75,Ku76,Kh78}, mixed volumes made their appearance in algebraic geometry as intersection numbers of hypersurfaces defined by generic polynomials with prescribed configuration of the exponent vectors (Newton polytopes). To make things more precise, consider a bivariate Laurent polynomial 
$$f(x,y)=\sum_{(a,b)\in A}\lambda_{(a,b)}x^{a}y^{b},$$
where $A$ is a finite subset of $\Z^2$ and $\lambda_{(a,b)}$ are non-zero complex numbers. The set $A$ is called the support of $f$ and its convex hull is called the Newton polytope $P_f$ of $f$. According to the BKK Theorem \cite[Sec 7.5]{CLO05}, given two Laurent polynomials $f,g$ with Newton polytopes $P_f$ and $P_g$ and generic coefficients, the number of common zeros of $f$ and $g$ in $(\C^*)^2$ equals twice the {\it mixed area} $\vv(P_f,P_g)$. The latter can be computed as
\begin{equation}\label{e:mv}
 \vv(P_f,P_g)=\frac{1}{2}\left(\vol_2(P_f+P_g)-\vol_2(P_f)-\vol_2(P_g)\right), 
\end{equation}
where $P_f+P_g$ is the Minkowski sum of the polytopes and $\vol_2$ denotes the area, i.e. the 2-dimensional Euclidean volume.
In other words, $2\vv(P_f,P_g)$ is the (global) intersection number $(\G_f\cdot\G_g)$ of two curves $\G_f$ and $\G_g$ in $(\C^*)^2$ defined by $f$ and $g$. 

This paper is motivated by the following question. Let $\G_1,\dots, \G_n$ be $n$ curves in $(\C^*)^2$
defined by generic Laurent polynomials $f_1,\dots, f_n$ with Newton polytopes $P_1,\dots, P_n$.
They define ${n\choose 2}$ pairwise intersection numbers $(\G_i\cdot\G_j)$ for $1\leq i<j\leq n$. What are possible relations between these numbers? Equivalently, what are possible relations between
${n\choose 2}$ mixed areas $\vv(P_i,P_j)$ defined by arbitrary collections $P_1,\dots,P_n$
of $n$ lattice polytopes  in $\R^2$?

In \cite{AS22} it was shown that starting with $n=4$ there are algebraic relations between the ${n\choose 2}$ mixed areas of $n$ lattice polytopes  (in fact, arbitrary convex bodies) in $\R^2$. In particular, for $n=4$, the six mixed areas $v_{ij}=\vv(P_i,P_j)$ for $1\leq i<j\leq 4$ satisfy three quadratic inequalities, called the Plücker-type inequalities:
\begin{equation*}
  v_{12}v_{34}+v_{13}v_{24}\geq v_{14}v_{23},\quad v_{13}v_{24}+v_{14}v_{23}\geq v_{12}v_{34}, \quad v_{14}v_{23}+v_{12}v_{34}\geq v_{13}v_{24}.  
\end{equation*}
Moreover, the Plücker-type inequalities, together with the non-negativity of the $v_{ij}$, completely describe the set of all possible 6-tuples of mixed areas of 4 convex bodies in $\R^2$, \cite[Th 3.2]{AS22}.
However, if we restrict ourselves to the case of lattice polytopes, will there be additional relations besides the Plücker-type inequalities and the necessary condition $2v_{ij}\in\Z_{\geq 0}$? In this paper we show that, although there are no additional algebraic relations between the $v_{ij}$, there are some arithmetic conditions the six integers $2v_{ij}$ must satisfy. %, see \rt{} and \rt{}.

The question of describing all possible relations between mixed volumes of a collection of convex bodies goes back to the results of Heine \cite{heine1938wertvorrat} and Shephard \cite{Shephard1960}. Let $K_1,\dots, K_n$ be convex bodies in $\R^d$
and consider their Minkowski linear combination $x_1K_1+\dots+x_nK_n$ for some non-negative scaling factors $x_1,\dots, x_n$. Then, as shown by Minkowski \cite{Min03}, the volume of $x_1K_1+\dots+x_nK_n$ is a homogeneous degree $d$ polynomial in $x_1,\dots,x_n$ with mixed volumes as coefficients:
\begin{equation}\label{e:volume-poly}
  \vol_d(x_1K_1+\dots+x_nK_n)=\sum_{1\leq i_1,\dots,i_d\leq n}\!\vv(K_{i_1},\dots, K_{i_d})x_{i_1}\cdots x_{i_d}.  
\end{equation}
(Note that we allow repetitions of the indices in the above sum.)
For example, for $n=d=2$, using the symmetry of the mixed volume and the fact that $\vv(K,K)=\vol_2(K)$, we obtain the following quadratic form
\begin{equation}\label{e:volume-poly-dim2}
\vol_2(x_1K_1+x_2K_2)=\vol_2(K_1)x_1^2+2\vv(K_1,K_2)x_1x_2+\vol_2(K_2)x_2^2.
\end{equation}
The distinctive feature of this form is that it has non-negative coefficients and is {\it indefinite}. The latter is provided by the Minkowski inequality \cite[Th 7.2.1]{Schneider2014}
\begin{equation}\label{e:Mink-ineq}
\vol_2(K_1)\vol_2(K_2)\leq \vv(K_1,K_2)^2.
\end{equation}

The Heine-Shephard problem asks for a description of the space of all volume polynomials \re{volume-poly}, which we denote by $\cV_{n,d}$, in terms of polynomial inequalities on their coefficients. For example, the space $\cV_{2,2}$ of volume polynomials  of pairs of planar bodies \re{volume-poly-dim2} is precisely the space of indefinite quadratic forms with non-negative coefficients. Apart from the case of $n=d=2$ above, the only known descriptions of $\cV_{n,d}$ are by Heine for $n=3$ and $d=2$ and by Shephard for $n=2$ and arbitrary $d\geq 3$, see \cite{heine1938wertvorrat, Shephard1960}. 

The question we consider in this paper is related to the {\it discrete} version of the Heine-Shephard problem, in which the class of all convex bodies is replaced with the class of lattice polytopes. The smallest nontrivial case of the discrete Heine-Shephard problem for $n=d=2$ was recently solved in \cite{MixedLattice24}. Namely, it was shown that the space
$\cV_{2,2}(\Z)$ of volume polynomials of pairs of planar lattice polytopes equals the space of indefinite quadratic forms with non-negative integer coefficients. In this paper, we consider $n=3,4$ and $d=2$, but we focus on the {\it square-free} part of the volume polynomial $\vol_2(x_1K_1+\dots+x_nK_n)$, i.e.
$$v(x_1,\dots, x_n)=\sum_{1\leq i<j\leq n}\! 2\vv(K_i,K_j)x_ix_j.$$ 
Let $\SFV_{n,2}$ denote the space of all such polynomials and $\SFV_{n,2}(\Z)$ the subset of those corresponding to lattice polytopes. For $n=3$ the situation is simple: $\SFV_{3,2}(\Z)$ consists of all trivariate square-free quadratic forms with non-negative integer coefficients, see \rp{sf-3}. For $n=4$, as we pointed out above, $\SFV_{4,2}(\Z)$ is a subset of the {\it discrete Plücker set} $\Pl_{4,2}(\Z)$ of  square-free quadratic forms whose coefficients are non-negative integers and satisfy Plücker-type inequalities
$$\Pl_{4,2}(\Z)=\Big\{\sum_{1\leq i<j\leq 4}\!v_{ij}x_ix_j\, :\, v_{ij}\in \Z_{\geq 0},\,  v_{ij}v_{kl}+v_{ik}v_{jl}\geq v_{il}v_{jk},\text{ for }\{i,j\}\sqcup\{k,l\}=[4]\Big\}.$$
In \rs{boundary} we show that every boundary point of $\Pl_{4,2}(\Z)$ (i.e. where at least one of the Plücker-type inequalities becomes equality) belongs to $\SFV_{4,2}(\Z)$. In Sections \ref{S:interior-in} and \ref{S:interior-out} we describe families of polynomials in the interior of $\Pl_{4,2}(\Z)$ that lie inside and outside of $\SFV_{4,2}(\Z)$. Our analysis is based on understanding a ``discrete diagram'' which is the image of the map that sends a pair of planar lattice polytopes to the triple consisting of their lattice widths (in some direction) and their mixed area. We do this in \rs{diagram}.

Another version of the Heine-Shephard problem, also motivated by a question in intersection theory, was recently considered in \cite{ProjAreas}. There, the authors study the problem of describing all relations among the areas of the projections of a convex body $K \subset \mathbb{R}^4$ onto the six two-dimensional coordinate subspaces. These projections can be expressed as mixed volumes $\vv(K,K,I_i,I_j)$ for $1 \le i < j \le 4$, where $I_1,\dots,I_4$ are the coordinate unit segments. Thus, the problem in \cite{ProjAreas} amounts to describing part of the volume polynomial of $(K,I_1,\dots,I_4)$, with the segments fixed and $K$ varying. As in the present paper, the authors show a discrepancy between the continuous and discrete cases.

\subsection*{Acknowledgments} The first author was supported by an Undergraduate Summer Research Award (USRA) funded by the Office of Research at Cleveland State University and held during Summer 2025. The second author is supported by an AMS–Simons Travel Grant. The authors are grateful to Gennadiy Averkov for his interest in this work.

\section{Preliminaries}
Let $\R^d$ be a real $d$-dimensional vector space and $\Z^d\subset\R^d$ the integer lattice. 
%By $\Z_{\geq n}$ we denote the set of all integers greater than or equal to $n$.  
We use $\{e_1,\dots,e_n\}$ to denote the standard basis for $\R^d$ and $\la u,v\ra$ the standard Euclidean inner product. 
For any $A,B\subset\R^d$, their {\it Minkowski sum} $A+B$ is the pointwise sum $A+B=\{a+b\in\R^d : a\in A, b\in B\}$.

Recall that a {\it lattice polytope} in $\R^d$ is the convex hull of a finite subset of $\Z^d$. We let $\cP(\Z^d)$ denote the set of all lattice polytopes in $\R^d$. The set $\cP(\Z^d)$ is invariant under the group $\AGL_d(\Z)$ of {\it affine unimodular transformations} consisting of maps $F:\R^d\to\R^d$ of the form $F(x)=Ax+b$ for some unimodular
matrix $A\in\GL_d(\Z)$ and a lattice vector $b\in\Z^d$. 

As common in lattice geometry, we normalize the Euclidean volume of lattice polytopes by a factor of $d!$ and denote it by $\Vol_d(P)$.
This way for every $P\in\cP(\Z^d)$, the normalized volume $\Vol_d(P)$ is a non-negative integer. For example, for the standard $d$-simplex $\D_d=\conv\{0,e_1,\dots,e_d\}$ we have $\Vol_d(\D_d)=1$.  
Consequently, the {\it normalized mixed volume}, defined by 
\begin{equation}\label{e:polarization}
\V(P_1,\dots,P_d)=\frac{1}{d!}\sum_{\ell=1}^d(-1)^{d+\ell}\!\sum_{i_1<\dots<i_\ell}\!\Vol_d(P_{i_1}+\dots+P_{i_\ell}),
\end{equation}
is an integer for any $P_1,\dots, P_d\in\cal P(\Z^d)$.

We list some fundamental properties of the mixed volumes that we will use throughout the paper: $\V(P_1,\dots,P_d)$ is (a) non-negative; (b) monotone in each entry with respect to inclusion; (c) invariant under independent lattice translations of the $P_i$; (d) invariant under simultaneous affine unimodular transformations of the $P_i$; and (e) Minkowski-linear in each entry, i.e.
$$\V(P_1,\dots,\lambda P_i+\mu P_i',\dots, P_d)=\lambda\V(P_1,\dots,P_i,\dots, P_d)+\mu\V(P_1,\dots,P_i',\dots, P_d)$$
for non-negative $\lambda, \mu\in\R$.

In the 2-dimensional case, which is the focus of this paper, \re{polarization} simplifies to
\begin{equation}\label{e:polarization-2}
\V(P,Q)=\frac{1}{2}\left(\Vol_2(P+Q)-\Vol_2(P)-\Vol_2(Q)\right).
\end{equation}
Note that $\V(P_1,P_2)=2\vv(P_1,P_2)$ relates the normalized mixed area and the (usual) mixed area introduced in \re{mv}. An immediate consequence of \re{polarization-2} is the formula for the mixed area of segments. Given a pair of lattice segments $I=\conv\{(0,0),(a,b)\}$ and $J=\conv\{(0,0),(c,d)\}$ their normalized mixed area equals the absolute value of the corresponding determinant $\V(I,J)=|ad-bc|$. In particular, $\V(I,J)=0$ when $I$ and $J$ are parallel. It is not hard to see from the monotonicity property that $\V(P,Q)=0$ if and only if either one of $P,Q$ is a point or $P$ and $Q$ are parallel segments.

We will need another formula for the normalized mixed area in terms of the support function of $P$ and the boundary of $Q$, see \cite[Sec. 2.5]{mvClass2019}. Recall that a vector $u\in\Z^2$ is called {\it primitive} if its coordinates are relatively prime. Furthermore, for a lattice segment $I\subset\Z^2$ define its {\it lattice length}
$\Vol_1(I)$ to be $|I\cap\Z^2|-1$. In particular, 
if $I=\conv\{(0,0),(a,b)\}$ then $\Vol_1(I)=\gcd(a,b)$.
Lattice segments of lattice length one are called primitive. Finally, let $h_P(u)=\max\{\la x,u \ra : x\in P\}$ be the support function of $P$. We have 
\begin{equation}\label{e:support-function}
\V(P,Q)=\sum_{u\in U_Q}h_P(u)\Vol_1(Q^u),
\end{equation}
where $U_Q$ is the set of primitive outer normals to the sides of $Q$ and $Q^u$ is the side of $Q$ with the outer normal $u$. In practice, it is more convenient to express \re{support-function} using the {\it normalized outer normals} $\Vol_1(Q^u)u$:
\begin{equation}\label{e:support-function-normalized}
\V(P,Q)=\sum_{u\in U_Q}h_P(\Vol_1(Q^u)u).
\end{equation}

\begin{Ex}
Let $P$ be the unit square and $Q=\conv\left\{(0,0),(0,2),(3,3)\right\}$ a triangle as in \rf{example}. Then $Q$ has three normalized outer normals 
$$u_1=\left(\begin{matrix}-2\\ \ \ 0\end{matrix}\right),\ \ u_2=\left(\begin{matrix}-1\\ \ \ 3\end{matrix}\right), \ \ u_3=\left(\begin{matrix}\ \ 3\\ -3\end{matrix}\right).$$
\begin{figure}[h]
\begin{center}
\includegraphics[scale=.45]{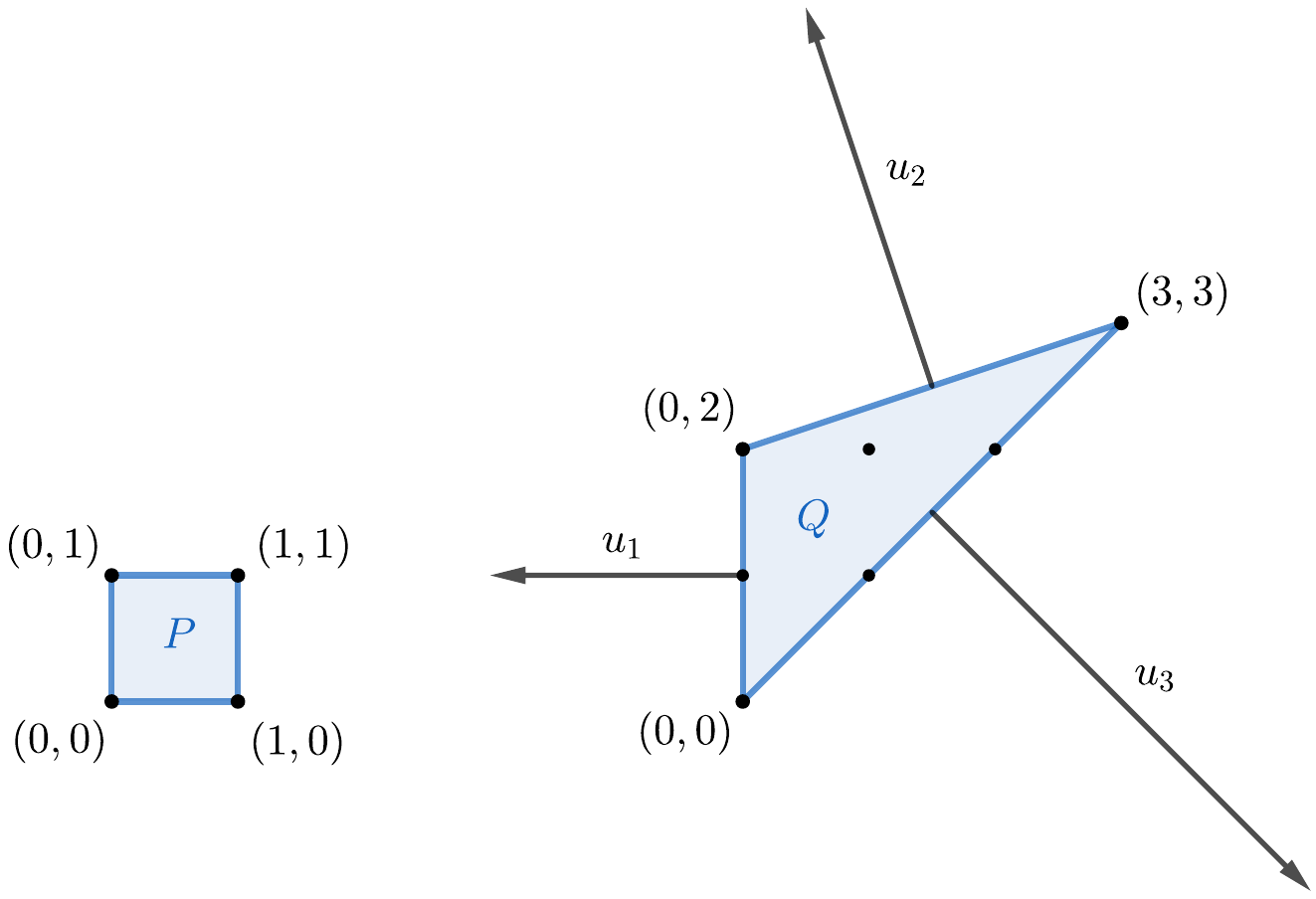} 
\caption{The vertices of $P$ and the normalized outer normals of $Q$}
\label{F:example} 
\end{center}
\end{figure}
We have 
$$\V(P,Q)= h_P(u_1)+h_P(u_2)+h_P(u_3)=0+3+3=6.$$ 
\end{Ex}

Let $P\in\cP(\Z^2)$ and $u\in\Z^2$ a primitive vector. Define the {\it lattice width} of $P$ in the direction of $u$ by
$$\w_u(P)=\max\{\la x,u\ra : x\in P\}-\min\{\la x,u\ra : x\in P\}.$$
For example, when $u=e_2$ the lattice width $\w_{e_2}(P)$ is the height of $P$, that is, the difference between the largest and the smallest values of the second coordinate of points of $P$. 
%In the following proposition we look at two special instances of \re{support-function} when $Q$ is a lattice segment and the standard 2-simplex.

\begin{Prop}\label{P:two-instances} Let $P\in\cP(\Z^2)$. Then
\begin{enumerate}
\item for a lattice segment $I\in\cP(\Z^2)$ we have $\V(P,I)=\w_{u}(P)\Vol_1(I)$,
    where $u$ is a primitive vector orthogonal to $I$;
\item $\V(P,\D_2)$ equals the smallest non-negative integer $\delta$ such that a lattice translate of $P$ is contained in $\delta\Delta_2$.
\end{enumerate}
\end{Prop}

\begin{pf}
Part (1) follows right away from \re{support-function} and the observation $\w_u(P)=h_P(u)+h_P(-u)$. 
For part (2), the formula \re{support-function} provides
$$\V(P,Q)=h_P(e_1+e_2)+h_P(-e_1)+h_P(-e_2).$$
After a lattice translation we may assume $h_P(-e_1)=h_P(-e_2)=0$. Then $h_P(e_1+e_2)=\max_{x\in P}(x_1+x_2)$ is precisely the smallest factor $\delta$ such that the dilate $\delta\D_2$ contains $P$.
\end{pf}

It is not hard to see that lattice polytopes $P\in\cP(\Z^d)$ of normalized volume one are $\AGL_d(\Z)$-equivalent to the standard $d$-simplex $\D_d$. In this case, $P$ is called a {\it unimodular $d$-simplex}. In \cite{Esterov2015} Esterov and Gusev classified all collections $P_1,\dots, P_d\in\cP(\Z^d)$ whose normalized mixed volume equals one. In the following proposition we present the 2-dimensional case of their result. For the sake of completeness, we provide a short proof.

\begin{Prop}\label{P:EG-2d} 
Let $P,Q$ be lattice polytopes in $\R^2$ with $\V(P,Q)=1$. Then either one of them is a primitive segment 
or both are lattice translates of the same unimodular triangle.
\end{Prop}

\begin{pf} Suppose neither $P$ nor $Q$ is a lattice segment. Then $\Vol_2(P)$ and $\Vol_2(Q)$ are positive integers.
By the Minkowski inequality \re{Mink-ineq},
$$1\leq \Vol_2(P)\Vol_2(Q)\leq \V(P,Q)^2=1.$$
This shows that $\Vol_2(P)=\Vol_2(Q)=1$ and, hence, $P$ and $Q$ are unimodular triangles. It remains to show that $P=Q$. After a unimodular transformation, we may assume that $Q=\Delta_2$. Then, by part (2) of \rp{two-instances}, $\V(P,\Delta_2)=1$ implies
that $P\subseteq\Delta_2$ (after a lattice translation) and, hence, $P=\Delta_2$, as $P$ is 2-dimensional.

Now assume that $P$ is a lattice segment. Then, by the linearity of the mixed area, $P$ has to be of lattice length one, i.e. a primitive segment.
\end{pf}

\section{The case of three planar lattice polytopes}\label{S:n=3}

%As follows from the next proposition, every homogeneous square-free quadratic polynomial with non-negative integer coefficients is the square-free part of a volume polynomial of three lattice polygons.
In this short section we show that, in general, there are no relations between the mixed areas of $\V(P_1,P_2)$, $\V(P_1,P_3)$, and $\V(P_2,P_3)$ defined by
three lattice polytopes $P_1,P_2,P_3\in\cP(\Z^2)$.

\begin{Prop}\label{P:sf-3}
Let $(v_{12},v_{13},v_{23})\in\Z_{\geq 0}^3$. Then there exist lattice polytopes $P_1,P_2,P_3$ in $\R^2$ such that
$v_{ij}=\V(P_i,P_j)$ for $1\leq i<j\leq 3$. Equivalently, 
$$\SFV_{3,2}(\Z)=\{v_{12}x_1x_2+v_{13}x_1x_3+v_{23}x_2x_3 : v_{ij}\in\Z_{\geq 0}\}.$$
\end{Prop}

\begin{pf} Assume $v_{12}$ and $v_{13}$ are not both zero and 
let $d=\gcd(v_{12},v_{13})$. Since $v_{12}/d$ and $v_{13}/d$ are relatively prime,
there exist $a,b\in\Z$ such that $av_{13}/d -bv_{12}/d=v_{23}$. Now the lattice segments
$P_1=\conv\{(0,0),(d,0)\}$, $P_2=\conv\{(0,0),(a,v_{12}/d)\}$, and $P_3=\conv\{(0,0),(b,v_{13}/d)\}$ satisfy 
the above requirement.

If $v_{12}=v_{13}=0$ we put $P_1=\{(0,0)\}$, $P_2=\conv\{(0,0),(1,0)\}$, and $P_3=\conv\{(0,0),(0,v_{23})\}$.

\end{pf}

\section{The discrete diagram $\left(\w_u(P),\V(P,Q), \w_u(Q)\right)$}\label{S:diagram}

Let $P\in\cP(\Z^2)$ be a lattice polytope and $u\in\Z^2$ a primitive vector. Recall that the lattice width of $P$ in the direction of $u$ is defined as $\w_u(P)=\max_{x\in P}\la x,u\ra-\min_{x\in P}\la x,u\ra$. We are interested in describing the image $\Phi:=\im\varphi$ of the map
$$\varphi:\cP(\Z^2)\times\cP(\Z^2)\to\Z^3,\quad (P,Q)\mapsto\left(\w_u(P),\V(P,Q), \w_u(Q)\right),$$
which we call a {\it discrete diagram}. The terminology comes from the work of Blaschke \cite{Blaschke1916} and Santal\'o \cite{Santalo61} who proposed the study of diagrams, i.e. images of maps that send a convex body (or a tuple of convex bodies) to a tuple of its geometric invariants, such as volume, surface area, and mean curvature. 

Note that the discrete diagram $\Phi$ does not depend on the choice of the primitive vector $u$. Indeed, for any primitive vectors $u,u'\in\Z^2$ there exists a unimodular matrix $A\in\GL_2(\Z)$ such that $u'=Au$. Consider the lattice polytopes $P'=PA^{-1}$ and $Q'=QA^{-1}$, where by definition 
$PA^{-1}=\{xA^{-1} : x\in P\}$. Then, by the invariance of the mixed area,
$$
\left(\w_u(P'),\V(P',Q'),\w_u(Q')\right)=
\left(\w_{u'}(P),\V(P,Q),\w_{u'}(Q)\right).
$$

\begin{Rem} It is easy to see that in the case of general convex bodies $K,L$ in $\R^2$ the ``continuous'' diagram 
$\left(\w_u(K),\V(K,L),\w_u(L)\right)$ coincides with $\R_{\geq 0}^3\setminus\left(0\times\R_{>0}\times 0\right)$. First, note that if $\w_u(K)=\w_u(L)=0$ then $K$ and $L$ must be line segments (possibly degenerated to a point) orthogonal to $u$ which implies $\V(K,L)=0$.
Now, every triple $(w_1,v,w_2)\in\R_{\geq 0}^3$ with $w_1+w_2>0$ can be realized as $\left(\w_u(K),\V(K,L),\w_u(L)\right)$ for some segments 
$K,L$. Indeed, without loss of generality, we may assume $u=e_2$ and $w_2>0$. Then the segments
$K=\conv\{(0,0),(v/w_2,w_1)\}$ and $L=\conv\{(0,0),(0,w_2)\}$
satisfy the required.
\end{Rem}

Now consider $P,Q\in\cP(\Z^2)$. Clearly, there are no restrictions on 
the lattice widths $\w_u(P)$ and $\w_u(Q)$ besides being non-negative integers. Therefore, to understand the discrete diagram $\Phi$ we fix the values of $\w_u(P)$ and $\w_u(Q)$ and look for possible values of $\V(P,Q)$. As above, if both widths are zero then so is the mixed area, so for the rest of the section we will assume that the widths are not both zero.
The next proposition shows that in general, given $\w_u(P)$ and $\w_u(Q)$, there is a gap in the possible values of $\V(P,Q)$. 

\begin{Prop}\label{P:gap} 
Let $P, Q\in\cP(\Z^2)$ and 
let $\w_u(P)$ and $\w_u(Q)$ be the lattice widths of $P$ and $Q$ in the direction of a primitive vector $u$.
Then either $\V(P,Q)=0$ or $\V(P,Q)\geq\min\left\{d,\frac{\w_u(P)}{d},\frac{\w_u(Q)}{d}\right\}$, where 
$d=\gcd(\w_u(P),\w_u(Q))$.
\end{Prop}

\begin{pf}
Applying an appropriate unimodular transformation, we may assume that $u=e_2$
and $y_P=\w_u(P)$, $y_Q=\w_u(Q)$ are the heights of $P$ and $Q$. Furthermore, after applying independent lattice translations, we may assume that $P$ and $Q$ contain lattice segments 
$I_P=\conv\{(0,0), (x_P,y_P)\}$ and $I_Q=\conv\{(0,0), (x_Q,y_Q)\}$, respectively, for some $x_P,x_Q\in\Z$.

By the monotonicity of the mixed volume,
$$\V(P,Q)\geq \V(I_P,I_Q)=|x_Py_Q-x_Qy_P|\in d\Z.$$
Assume $0<\V(P,Q)< d$. Then $x_Py_Q=x_Qy_P$, which implies that $\frac{y_P}{d}\mathrel{|} x_P$ and $\frac{y_Q}{d}\mathrel{|} x_Q$. Since $P, Q$ are not parallel segments (otherwise $\V(P,Q)=0$), either
$P$ has positive lattice width in the direction $u_Q$ orthogonal to $I_Q$ or $Q$ has positive lattice width in the direction $u_P$ orthogonal to $I_P$ (or both). Without loss of generality, we assume the former. Then, again by monotonicity and part (2) of \rp{two-instances}, we have
$$\V(P,Q)\geq \V(P,I_Q)=\w_{u_Q}(P)\Vol_1(I_Q)\geq \gcd(x_Q,y_Q)\geq \frac{y_Q}{d}.$$
The statement now follows.
\end{pf}

The pair of relatively prime integers $\left\{\frac{\w_u(P)}{d},\frac{\w_u(Q)}{d}\right\}$ from \rp{gap} generates a semigroup which turns out to be a part of the discrete diagram.

\begin{Prop}\label{P:semigroup} 
Let $w_1,w_2\in\Z_{\geq 0}$ not both zero and let $d=\gcd(w_1,w_2)$. Then for any $m_1,m_2\in\Z_{\geq 0}$ and primitive $u\in\Z^2$ there exist $P,Q\in\cP(\Z^2)$ such that $\w_u(P)=w_1$, $\w_u(Q)=w_2$, and 
$\V(P,Q)=m_1\frac{w_1}{d}+m_2\frac{w_2}{d}$. 
\end{Prop}

\begin{pf}
By the discussion above, we may assume that $u=e_2$. Let $I$ denote the segment $\conv\{(0,0),(1,0)\}$. By the division algorithm, there exist $q_1,r_1,q_2,r_2\in\mathbb{Z}$ where $0\leq r_1,r_2<d$ such that $m_1=q_1d+r_1$ and $m_2=q_2d+r_2$. Define $P_0=\text{conv}\{(0,0),(\frac{w_1}{d},w_1),(\frac{w_1}{d},w_1-r_2)\}$ and $Q_0=\text{conv}\{(0,0),(0,r_1),(\frac{w_2}{d},w_2)\}$, see \rf{semigroup}. 
\begin{figure}[h]
\begin{center}
\includegraphics[scale=.48]{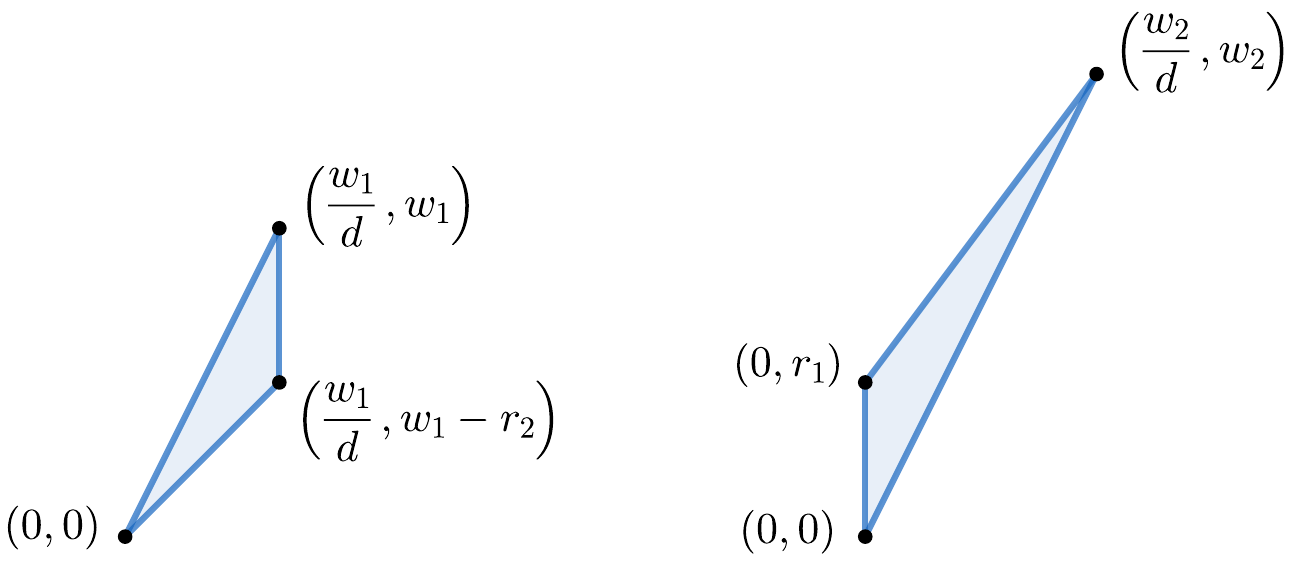} 
\caption{Lattice triangles $P_0$ and $Q_0$}
\label{F:semigroup} 
\end{center}
\end{figure}
To compute $\V(P_0,Q_0)$ we apply \re{support-function-normalized}. The normalized outer normals to $Q_0$ are 
$$u_1=\left(\begin{matrix}-r_1\\ 0\end{matrix}\right),\ \ u_2=\left(\begin{matrix}r_1-w_2\\ w_2/d\end{matrix}\right), \ \ u_3=\left(\begin{matrix}w_2\\ -w_2/d\end{matrix}\right).$$
Thus, we obtain
$$\V(P_0,Q_0)= h_P(u_1)+h_P(u_2)+h_P(u_3)=0+r_1\frac{w_1}{d}+r_2\frac{w_2}{d}.$$ Now, let $P=P_0+q_2I$ and $Q=Q_0+q_1I$. Note that, indeed, $\w_u(P)=w_1$ and $\w_u(Q)=w_2$. Also, by the multilinearity of the mixed area,
\begin{align*}
\V(P,Q)&=\V(P_0,Q_0)+q_1\V(P_0,I)+q_2\V(I,Q_0)+q_1q_2\V(I,I)\\&=r_1\frac{w_1}{d}+r_2\frac{w_2}{d}+q_1w_1+q_2w_2=m_1\frac{w_1}{d}+m_2\frac{w_2}{d}, 
\end{align*}
as desired. 
\end{pf}

One consequence of \rp{semigroup} is that, for given values of $\w_u(P)$ and $\w_u(Q)$,
the mixed volume $\V(P,Q)$ takes all possible values starting with 
$\left(\frac{\w_u(P)}{d}-1\right)\left(\frac{\w_u(Q)}{d}-1\right)$, by the Frobenius theorem (see, for example, \cite[Th 1.2]{BeckRobins}).
The next proposition further explores the saturation behavior of the values of $\V(P,Q)$.
Below we write ``$m\bmod n$'' for the remainder when $m$ is divided by $n$.

\begin{Prop}\label{P:mod-saturation} 
Let $(w_1,v,w_2)\in\Z_{\geq 0}^3$ such that
$v\geq\min\{w_1\bmod w_2,\  w_2\bmod w_1\}$. Then for any primitive $u\in\Z^2$ there exist
$P, Q\in\cP(\Z^2)$ such that $w_1=\w_{u}(P)$, $w_2=\w_{u}(Q)$,
and $v=\V(P,Q)$.
\end{Prop}
\begin{pf}
As in the previous proof, it is enough to consider $u=e_2$. As before, let $I=\conv\{(0,0),(1,0)\}$.
    Without loss of generality, we may assume $w_1\leq w_2$. By the division algorithm, there exist $q_2,r_2\in \mathbb{Z}$ with $0\leq r_2< w_1$ such that $w_2=q_2w_1+r_2$. Then the assumption $v\geq\min\{w_1\bmod w_2,\  w_2\bmod w_1\}$ translates to $v\geq r_2$.
     
    Take $P=\text{conv}\{(0,0),(1,w_1)\}$. Clearly, $\w_u(P)=\V(P,I)=w_1$. We have two cases depending on the order of $v$ and $w_1$. 

    \textbf{Case 1.} Suppose $v<w_1$. Define $Q=\text{conv}\{(0,0),(q_2,w_2),(q_2,w_2-v)\}$, see \rf{case1}. 
    \begin{figure}[h]
    \begin{center}
    \includegraphics[scale=.48]{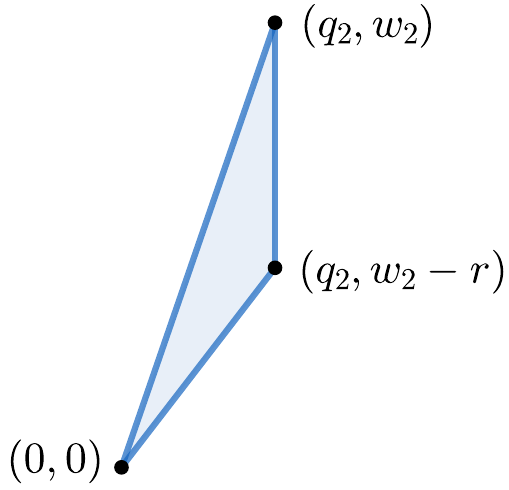} 
    \caption{Lattice triangle $Q$ in Case 1}
    \label{F:case1} 
    \end{center}
    \end{figure}
    Since $P$ is a primitive lattice segment, we can use part (1) of \rp{two-instances} to compute $\V(P,Q)$. Let $u_P=\left(\begin{matrix}-w_1\\ 1\end{matrix}\right)$ be a primitive vector orthogonal to $P$. Then 
    $$\V(P,Q)=\max_{x\in Q}\la x,u_P\ra-\min_{x\in Q}\la x,u_P\ra=(w_2-q_2w_1)-(w_2-v-q_2w_1)=r_2-(r_2-v)=v.$$
    Note that the inequalities $r_2\geq0$ and $r_2-v\leq 0$ justify that these are indeed the maximum and the minimum in the above computation.
    Furthermore, $\w_u(Q)=w_2$, since $w_2-v\geq w_1-v>0$.

    \textbf{Case 2.} Suppose $w_1\leq v$. Dividing $v$ by $w_1$ with the remainder, we get $v=qw_1+r$ for some $0\leq r<w_1$ and $q\geq 1$. Define $Q_0=\text{conv}\{(0,0),(q_2,w_2),(q_2+1,w_2-r)\}$, see \rf{case2}.
    \begin{figure}[h]
    \begin{center}
    \includegraphics[scale=.48]{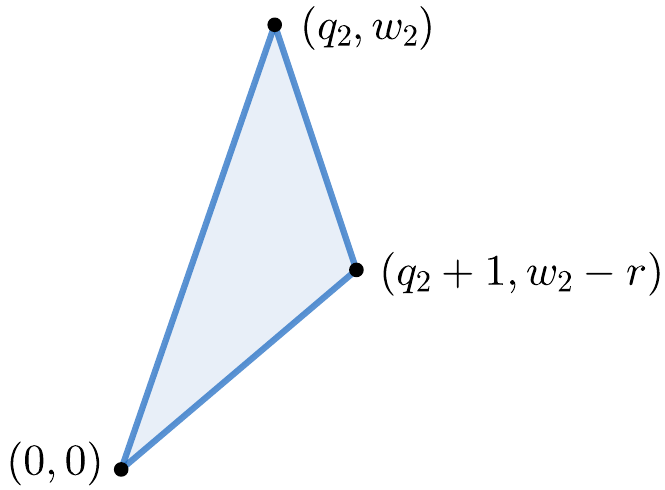} 
    \caption{Lattice triangle $Q_0$ in Case 2}
    \label{F:case2} 
    \end{center}
    \end{figure}
    Similarly to the previous case, we have 
    $$\V(P,Q_0)=\max_{x\in Q_0}\la x,u_P\ra-\min_{x\in Q_0}\la x,u_P\ra=r_2-(r_2-r-w_1)=r+w_1.$$
    Note that $r_2-r-w_1\leq 0$ is the minimum since $0\leq r,r_2<w_1$. Now for $Q=Q_0+(q-1)I$ we have
    $$\V(P,Q)=\V(P,Q_0)+(q-1)\V(P,I)=r+w_1+(q-1)w_1=v.$$ Finally, $\w_u(Q)=\w_u(Q_0)=w_2$, since $w_2-r> r_2-r>0$.  
\end{pf}

We summarize our results about the discrete diagram $\Phi$ in the following theorem. Given $(w_1,w_2)\in\Z_{\geq 0}^2$ we use $\Phi_{w_1,w_2}$ to denote the corresponding section of the diagram
$$\Phi_{w_1,w_2}:=\{v\in\Z_{\geq 0}\,:\, (w_1,v,w_2)\in\Phi\}.$$

\begin{Th}\label{T:diagram}
Let $w_1,w_2\in\Z_{\geq 0}$ not both zero and let $d=\gcd(w_1,w_2)$. %The discrete diagram $\im \Phi$ satisfies the following.
\begin{enumerate}
    \item (Gap) Let $\alpha=\min\left\{d,\frac{w_1}{d},\frac{w_2}{d}\right\}$. Then
    %The interval $\{1,\dots, \alpha-1\}$ is contained  
%    $$G_{w_1,w_2}=\left\{\left(w_1,v,w_2\right)\in\Z_{\geq 0}^3\, :\, 0< v<\alpha\right\}$$ lies 
$\{1,\dots, \alpha-1\}\subseteq\Z_{\geq 0}\setminus\Phi_{w_1,w_2}.$
%in the complement $\Z_{\geq 0}\setminus\Phi_{w_1,w_2}$.
    \item (Semigroup) %The semigroup
    %$$S_{w_1,w_2}=\left\{\Big(w_1,\ m_1\frac{w_1}{d}+m_2\frac{w_2}{d},\ w_2\Big)\, :\, m_1,m_2\in\Z_{\geq 0}\right\}$$
Let $S_{w_1,w_2}=\left\{m_1\frac{w_1}{d}+m_2\frac{w_2}{d}\, :\, m_1,m_2\in\Z_{\geq 0}\right\}$. Then
$S_{w_1,w_2}\subseteq\Phi_{w_1,w_2}.$
    %is contained in $\Phi_{w_1,w_2}$.
    \item (Saturation) %Let $\beta=\min\left\{\left(\frac{\w_u(P)}{d}-1\right)\left(\frac{\w_u(Q)}{d}-1\right),\  w_1\bmod w_2,\  w_2\bmod w_1\right\}$. 
    Let $\beta=\min\left\{w_1\bmod w_2,\  w_2\bmod w_1\right\}$. 
    Then $\Z_{\geq\beta}\subseteq\Phi_{w_1,w_2}.$
    %$$N_{w_1,w_2}=\left\{\left(w_1,v,\ w_2\right)\in\Z_{\geq 0}^3\, :\, v\geq \beta\right\}$$ 
    %$\{v\in\Z_{\geq 0}\,:\,v\geq\beta\}$ is contained in $\Phi_{w_1,w_2}$.
%$$\{v\in\Z_{\geq 0}\,:\,v\geq\beta\}\subseteq\Phi_{w_1,w_2}.$$
\end{enumerate}    
\end{Th}

\begin{Rem}
    We remark that the above theorem does not give a complete description of the discrete diagram, in general. For example, consider $(w_1,w_2)=(30,50)$ and so $d=10$, $\alpha=3$, $\beta=20$. According to \rt{diagram}, the section $\Phi_{30,50}$ contains the semigroup
    $S_{30,50}=\{0,3, 5, 6\}\cup\Z_{\geq 8}$. In fact, this containment is proper: we claim that $\Phi_{30,50}=\{0,3\}\cup\Z_{\geq 5}$.
    
To show that $7\in\Phi_{30,50}$, one can take the following polytopes with $\w_{e_2}(P)=30$ and $\w_{e_2}(Q)=50$:
$$P=\conv\{(0, 0), (3, 30), (1, 11), (0, 1)\},\quad \text{and}\quad Q=\conv\{(0, 0), (5, 50), (3, 31), (0, 2)\}.$$ A direct application of \re{support-function-normalized} verifies that $\V(P,Q)=7$. 

To show that $4\not\in\Phi_{30,50}$, we use the classification of all pairs $P,Q\in\cP(\Z^2)$ with mixed area 4, see \cite{mvClass2019, Esterov2016} and the database at \url{https://github.com/christopherborger/mixed_volume_classification}. One can split the set of such (unordered) pairs $P,Q$ into three families: 
\begin{enumerate}
    \item (infinite family) $P$ is a lattice segment with $\Vol_1(P)=1,2,4$ and $Q$ has width $\w_{u_P}(Q)=4,2,1$, respectively, where $u_P$ is a primitive vector orthogonal to $P$;
    \item (homothetic) $Q$ is an integer multiple of $P$ (10 pairs up to $\AGL_2(\Z)$-equivalence);
    \item (non-homothetic) $P$ and $Q$ are not homothetic (3 pairs up to $\AGL_2(\Z)$-equivalence, depicted in \rf{non-homothetic}).
    \begin{figure}[h]
    \begin{center}
    \includegraphics[scale=.45]{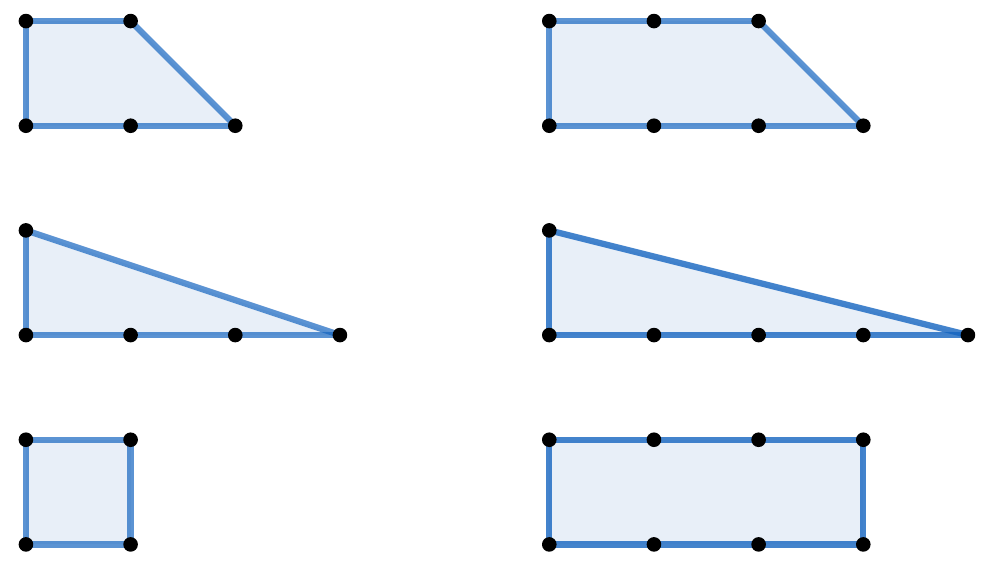} 
    \caption{Three pairs of non-homothetic lattice polygons of mixed area 4}
    \label{F:non-homothetic} 
    \end{center}
    \end{figure}
\end{enumerate} 
To account for the symmetry of the mixed volume, we need to show that for any $P,Q$ as above, $(\w_u(P),\w_u(Q))$ cannot equal $(30,50)$ or $(50,30)$. 

The second family is easy to discard: if $Q=\lambda P$ for some $\lambda \in\N$ then $50=\w_u(Q)=\lambda\w_u(P)=30\lambda$ which is impossible. Similarly, $(\w_u(P),\w_u(Q))\neq (50,30)$.

To deal with the infinite family, after an affine unimodular transformation, we may assume that $P=\conv\{(0,0),(m,0)\}$ for $m=1,2$. (Note that $P$ cannot have lattice length 4 since $4\not{|}\w_u(P)=30$.) If $\V(P,Q)=4$ for some $Q\in\cP(\Z^2)$ then $Q$ has height $4/m$.
Furthermore, if $\w_u(P)=30$ for some primitive $u\in\Z^2$ then $u=(30/m,k)$ where  $\gcd(30/m,k)=1$. Let $(a,b), (c,d)\in Q$ be the points where $\la x,u\ra$ attains the maximum and the minimum on $Q$. Then 
\begin{equation}\label{e:contradiction}
50=\w_u(Q)=\frac{30}{m}(a-c)+k(b-d),
\end{equation}
which implies that $5$ divides $b-d$. However, this is impossible, since $0\leq b,d\leq 4/m$ and $b=d$ would imply $3|50$ in \re{contradiction}. The case when $(\w_u(P),\w_u(Q))=(50,30)$ is completely analogous.

Finally, to take care of the remaining three pairs of non-homothetic polytopes $P,Q$ we do the following. Given a primitive $u\in\Z^2$, the condition $\w_u(P)=w$ can be interpreted as $|\la x-x',u\ra|\leq w$ for all pairs of vertices $x,x'\in P$ with at least one pair giving equality. Therefore, the set of $u$ with $\w_u(P)=w$ is the set of primitive points on the boundary of a rational polygon $R_P$ given by finitely many inequalities of the above form. In particular, the condition
$(\w_u(P),\w_u(Q))=(30,50)$ (or $(50,30)$) corresponds to primitive points in the intersection of two boundaries $\partial R_P\cap\partial R_Q$. We check that in all cases this intersection is empty except for the case when $P$ is a square and $Q$ is a rectangle as in \rf{non-homothetic}.
%when $P=\conv\{(0,0),(1,0),(0,1),(1,1)\}$ and $Q=\conv\{(0,0),(3,0),(0,1),(3,1)\}$. 
In this case $\partial R_P\cap\partial R_Q$ consists of four points $\pm(10,20)$, $\pm(10,-20)$ none of which is primitive. This is illustrated in \rf{polyhedra}.
    \begin{figure}[h]
    \begin{center}
    \includegraphics[scale=.4]{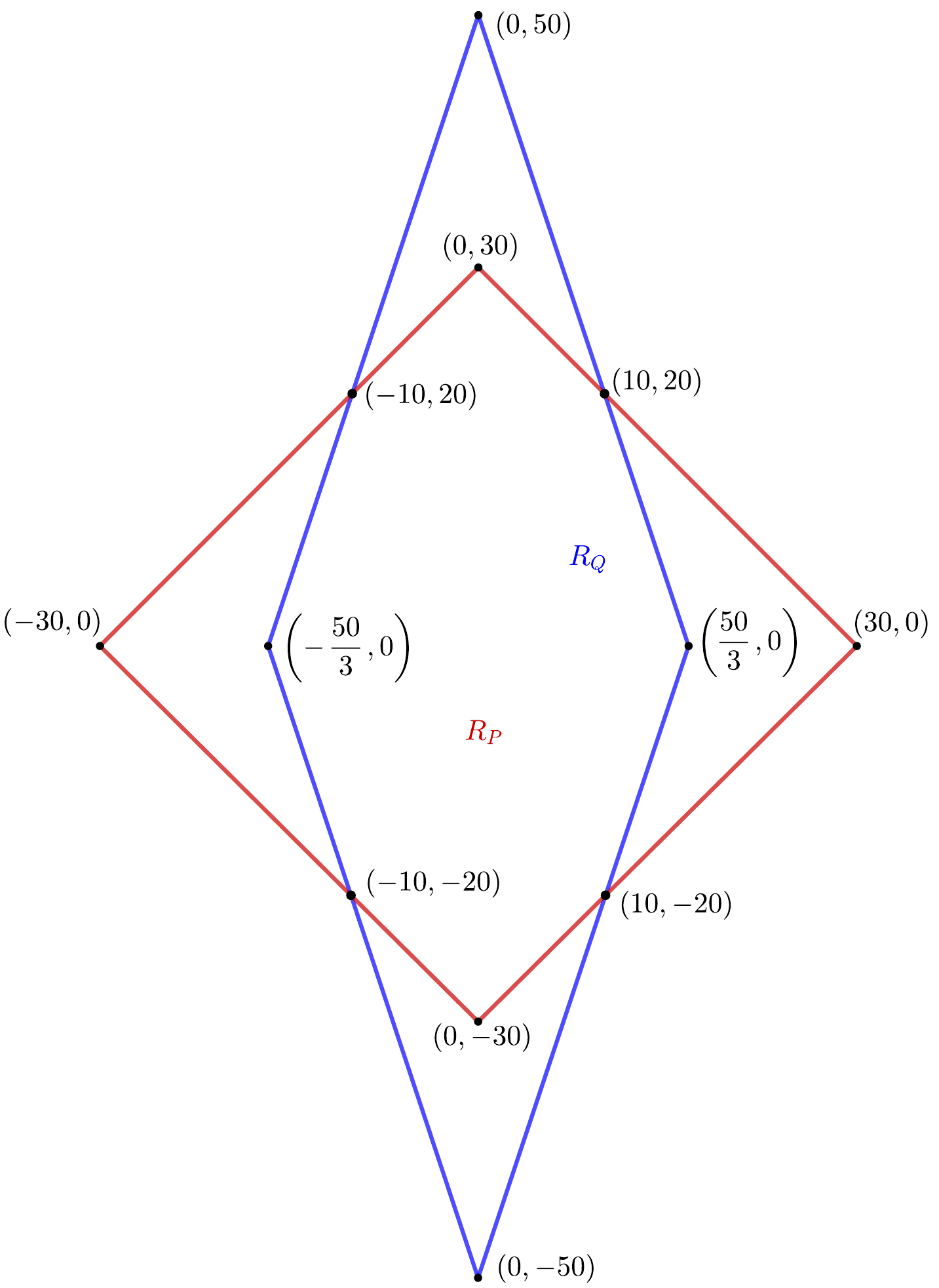} 
    \caption{Two rational polygons $R_P$ and $R_Q$}
    \label{F:polyhedra} 
    \end{center}
    \end{figure}
\end{Rem}

\section{The sets $\SFV_{4,2}(\Z)$ and $\Pl_{4,2}(\Z)$}

Recall from the introduction that the space of square-free parts of volume polynomials of four planar convex bodies
$$\SFV_{4,2}:=\Big\{\sum_{1\leq i<j\leq 4}\!2\vv(K_i,K_j)x_ix_j\,:\, K_1,\dots, K_4\text{ convex bodies in }\R^2\Big\}$$
has a semialgebraic description in terms of the Plücker-type inequalities, \cite[Th 3.2]{AS22}. Namely, $\SFV_{4,2}$ coincides with the {\it Plücker space}
\begin{equation*}\label{e:semialgebraic}
  \Pl_{4,2}:=\Big\{\sum_{1\leq i<j\leq 4}\!v_{ij}x_ix_j\,:\, v_{ij}\in \R_{\geq 0},\, v_{ij}v_{kl}+v_{ik}v_{jl}\geq v_{il}v_{jk},\text{ for }\{i,j\}\sqcup\{k,l\}=[4]\Big\}.  
\end{equation*}
We are interested in the discrete counterpart of $\SFV_{4,2}$ when the convex bodies are lattice polytopes
$$\SFV_{4,2}(\Z):=\Big\{\sum_{1\leq i<j\leq 4}\V(P_i,P_j)x_ix_j\,:\, P_1,\dots,P_4\in\cP(\Z^2)\Big\}.$$
Its elements are square-free quadratic forms with nonnegative integer coefficients satisfying the Plücker-type inequalities. In other words, $\SFV_{4,2}(\Z)$ is a subset of the set
$$\Pl_{4,2}(\Z):=\Big\{\sum_{1\leq i<j\leq 4}\!v_{ij}x_ix_j\, :\, v_{ij}\in \Z_{\geq 0},\,  v_{ij}v_{kl}+v_{ik}v_{jl}\geq v_{il}v_{jk},\text{ for }\{i,j\}\sqcup\{k,l\}=[4]\Big\},$$
which we call the {\it discrete Plücker set}. As we show in \rs{interior-out} the inclusion $\SFV_{4,2}(\Z)\subset \Pl_{4,2}(\Z)$ is, in fact, proper. 

It will be convenient for us to identify a square-free quadratic form $\sum_{1\leq i<j\leq 4} v_{ij}x_ix_j$ with its coefficient vector $v=(v_{ij} \,:\, 1\leq i<j\leq 4 )$ which we represent as a complete weighted graph $\bK_4$ on $\{1,\dots,4\}$ with the weights $v_{ij}$ assigned to its edges, see \rf{graph}. 
    \begin{figure}[h]
    \begin{center}
    \includegraphics[scale=.48]{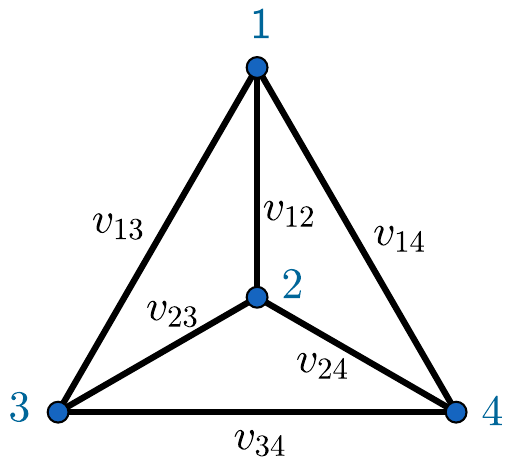} 
    \caption{A complete weighted graph representing a point in $\Pl_{4,2}(\Z)$}
    \label{F:graph} 
    \end{center}
    \end{figure}
Furthermore, we will use graph-theoretic terminology when we speak about elements of $\Pl_{4,2}$. For example, we say ``adjacent coordinates'' of $v\in\Pl_{4,2}$ meaning that the corresponding edges of the complete weighted graph associated with $v$ are adjacent.

There is a natural action of $\Sym_4$ on square-free quadratic forms by permuting the variables $x_1,\dots,x_4$. This corresponds to the $\Sym_4$-action on the coefficient vector $\sigma\cdot v=(v_{\sigma(i)\sigma(j)} \,:\, 1\leq i<j\leq 4 )$. It can be easily seen that $\SFV_{4,2}$, $\Pl_{4,2}$
and their discrete counterparts are invariant under this action. There is also the action of $\R_{>0}^4$ on $\SFV_{4,2}$ given by rescaling the bodies $K_1,\dots,K_4$ by positive real numbers. Analogously, rescaling lattice polytopes $P_1,\dots,P_4$ by positive integers defines a map from $\SFV_{4,2}(\Z)$ to itself. In coordinates, this map is given by
$$v\mapsto\lambda\cdot v,\ \ \text{ where }\ \lambda\cdot v=(\lambda_i\lambda_jv_{ij} \,:\, 1\leq i<j\leq 4 ).$$
for $v\in\SFV_{4,2}(\Z)$ and $\lambda\in\N^4$.

\begin{Def}
    We say that $v\in\Pl_{4,2}(\Z)$ is {\it $\bK_4$-primitive} if at each node the three adjacent coordinates of $v$ are relatively prime, i.e., for every $i\in[4]$ we have $\gcd(v_{ij} : j\in[4], j\neq i )=1$.
\end{Def}

The action of $\R_{>0}^4$ on $\SFV_{4,2}$ allows to reduce the proof of $\SFV_{4,2}=\Pl_{4,2}$ to the case when all but two adjacent coordinates of $v\in\SFV_{4,2}$ are equal to 1, see \cite[Sec 5]{AS22}. The following proposition gives a discrete analog of this statement. 

\begin{Prop}\label{P:reduction}
    Let $w\in\Pl_{4,2}(\Z)$ such that no three adjacent coordinates of $w$ are zero. Then $w=\lambda\cdot v$ for some $\lambda\in\N^4$ and $\bK_4$-primitive $v\in\Pl_{4,2}(\Z)$. Moreover, if $v$ lies in $\SFV_{4,2}(\Z)$ then so does $w$.
\end{Prop}

\begin{pf} Let the coordinates of $\lambda\in\N^4$ be $\lambda_1=\gcd(w_{12}, w_{13}, w_{14})$, 
     $\lambda_2=\gcd(w_{12}/\lambda_1, w_{23}, w_{24})$,
     $\lambda_3=\gcd(w_{13}/\lambda_1, w_{23}/\lambda_2, w_{34})$, and
     $\lambda_4=\gcd(w_{14}/\lambda_1, w_{24}/\lambda_2, w_{34}/\lambda_3)$. Define $v_{ij}=w_{ij}/\lambda_i\lambda_j$ for $1\leq i<j\leq 4$. Then $v$ is an integer vector satisfying $\lambda\cdot v=w$. One can readily see that $v$ is $\bK_4$-primitive. 
     Finally, if there exist $P_1,\dots, P_4\in\cP(\Z^2)$ such that
     $v_{ij}=\V(P_i,P_j)$ for $1\leq i<j\leq 4$ then $w_{ij}=\V(\lambda_iP,\lambda_jP_j)$ and, hence, $w\in\SFV_{4,2}(\Z)$.
\end{pf}

\section{The boundary of $\Pl_{4,2}(\Z)$}\label{S:boundary}
In this section we show that every boundary point of the discrete Plücker set $\Pl_{4,2}(\Z)$ is realizable as a point in $\SFV_{4,2}(\Z)$. By definition,
we say that $v\in\Pl_{4,2}(\Z)$ is a {\it boundary point}
if $v$ lies in the boundary of $\Pl_{4,2}$, i.e., if at least one of the Plücker-type inequality becomes equality at $v$; otherwise we say $v$ is an {\it interior point}. In fact, we will see that every boundary point of $\Pl_{4,2}(\Z)$ can be realized using four (possibly degenerate to a point) lattice segments.

We will need the following elementary property of the gcd.
\begin{Lem}\label{L:simpleNT} Let $a_1,a_2,b\in\Z$. If $\gcd(a_1,a_2,b)=1$ then
$$\gcd\left(\frac{b}{\gcd(a_1,b)},\frac{b}{\gcd(a_2,b)}\right)=\frac{b}{\gcd(a_1,b)\gcd(a_2,b)}.$$ 
\end{Lem}
\begin{pf} Denote $d_1=\gcd(a_1,b)$ and $d_2=\gcd(a_2,b)$. Since $\gcd(a_1,a_2,b)=1$ we have $\gcd(d_1,d_2)=1$.
This implies that $(d_1d_2)|b$. Now
$$\gcd\left(\frac{b}{d_1},\frac{b}{d_2}\right)=\gcd\left(d_2\frac{b}{d_1d_2},d_1\frac{b}{d_1d_2}\right)=\frac{b}{d_1d_2}\gcd(d_1,d_2)=\frac{b}{d_1d_2}.$$
\end{pf}

\begin{Th} Let $v=(v_{12},\dots, v_{34})$ be a boundary point of $\Pl_{4,2}(\Z)$. Then there exist
lattice segments $P_1,\dots, P_4\in\cP(\Z^2)$ such that $v_{ij}=\V(P_i,P_j)$ for $1\leq i<j\leq 4$.
\end{Th}

\begin{pf} Since $\Pl_{4,2}(\Z)$ is invariant under the $\Sym_4$-action described in the previous section, we may assume
without loss of generality that $v$ lies on the part of the boundary given by
\begin{equation}\label{e:boundary}
v_{13}v_{24}=v_{14}v_{23}+v_{12}v_{34}.
\end{equation}
Suppose $v$ has three adjacent zero coordinates, say, $v_{12}=v_{13}=v_{14}=0$. By \rp{sf-3} there exist $P_2,P_3,P_4\in\cP(\Z^2)$ such that $v_{ij}=\V(P_i,P_j)$ for $2\leq i<j\leq 4$. Then, letting $P_1=\{(0,0)\}$ we see that
$0=v_{1j}=\V(P_1,P_j)$ for $j=2,3,4$, as well.

Now suppose that $v$ does not have three adjacent zero coordinates. By \rp{reduction} we may assume that $v$ is $\bK_4$-primitive. In particular, $\gcd(v_{12}, v_{13}, v_{14})=1$.

First, assume $v_{14}=0$. Then \re{boundary} becomes $v_{13}v_{24}=v_{12}v_{34}$ and, since
$\gcd(v_{12}, v_{13})=1$, we see that $v_{24}=cv_{12}$ and $v_{34}=cv_{13}$ for some $c\in\N$. 
Also, there exist $a,b\in\Z$ such that $av_{13}-bv_{12}=v_{23}$.
We put $$P_1=\conv\{0,e_1\},\quad P_2=\conv\{0,ae_1+v_{12}e_2\},\quad P_3=\conv\{0,be_1+v_{13}e_2\},\quad P_4=\conv\{0,ce_1\}.$$
One can readily see that $v_{ij}=\V(P_i,P_j)$ for $1\leq i<j\leq 4$.

Now, assume $v_{14}>0$. We claim that there exist $a,b,c\in\Z$ such that the polytopes
$$P_1=\conv\{0,e_1\},\ P_2=\conv\{0,ae_1+v_{12}e_2\},\ P_3=\conv\{0,be_1+v_{13}e_2\},\ P_4=\conv\{0,ce_1+v_{14}e_2\}$$
satisfy $v_{ij}=\V(P_i,P_j)$ for $1\leq i<j\leq 4$, i.e., satisfy the linear system
$$\begin{cases}
av_{13}-bv_{12}=v_{23} &\\
av_{14}-cv_{12}=v_{24} &\\
bv_{14}-cv_{13}=v_{34}. &\\
\end{cases}$$
Over $\Q$ this system has a 1-parameter family of solutions
\begin{equation}\label{e:family}
a=\frac{cv_{12}+v_{24}}{v_{14}},\quad
b=\frac{cv_{13}+v_{34}}{v_{14}},\quad c\in\Q.
\end{equation}
Thus, we need to show that there exists $c\in\Z$ such that $v_{14}$ divides both $cv_{12}+v_{24}$ and $cv_{13}+v_{34}$.

First, we convert the fractions in \re{family} to lowest terms. For this we introduce
$$m_i=\frac{v_{14}}{\gcd(v_{1i},v_{14})},\quad \ u_{1i}=\frac{v_{1i}}{\gcd(v_{1i},v_{14})},\quad
\text{and}\quad 
\ u_{i4}=\frac{v_{i4}}{\gcd(v_{1i},v_{14})},\   \text{ for}\ \  i=2,3.$$
The fact that $u_{i4}$ are integers follows from \re{boundary} and the assumption $\gcd(v_{12}, v_{13}, v_{14})=1$.
Then our goal is to show the existence of $c\in\Z$ satisfying the system of congruences
$$\begin{cases}
cu_{12}+u_{24}\equiv 0 \;\bmod\; m_2 &\\
cu_{13}+u_{34}\equiv 0 \;\bmod\; m_3.
\end{cases}$$
Let $u_{1i}^{-1}$ denote the inverse of $u_{1i}$ modulo $m_i$, which exists since 
$\gcd(u_{1i},m_i)=1$, by construction. Then the above system can be written as
\begin{equation}\label{e:CRT}
\begin{cases}
c\equiv -u_{12}^{-1}u_{24} \;\bmod\; m_2 &\\
c\equiv -u_{13}^{-1}u_{34} \;\bmod\; m_3.
\end{cases}
\end{equation}

The existence of a solution to \re{CRT} now follows from the Generalized Chinese Remainder Theorem (see, for example, \cite[Th 3.12]{ENT98}). Indeed, let $g=\gcd(m_2,m_3)$. By \rl{simpleNT}, $$g=\frac{v_{14}}{\gcd(v_{12},v_{14})\gcd(v_{13},v_{14})}.$$
Thus \re{boundary} is equivalent to 
$u_{13}u_{24}=gv_{23}+u_{12}u_{34}$. This implies 
$u_{12}^{-1}u_{24} \equiv u_{13}^{-1}u_{34} \bmod g$, as required.

\end{pf}

\section{Realizable family in the interior of $\Pl_{4,2}(\Z)$}\label{S:interior-in}

In this section we describe an infinite family of interior points in the discrete Plücker set $\Pl_{4,2}(\Z)$ which can be realized as points in $\SFV_{4,2}(\Z)$.

\begin{Th}\label{T:two-ones} Let $v$ be an interior point of $\Pl_{4,2}(\Z)$ with two adjacent coordinates being 1. Then $v\in\SFV_{4,2}(\Z)$.  
\end{Th}

\begin{pf} As before, by the $\Sym_4$-invariance of $\Pl_{4,2}(\Z)$ we may assume that $v_{12}=v_{13}=1$. Then in the coordinates 
$x=v_{24}$ and $y=v_{34}$ the intersection $\Pl_{4,2}\cap\{v_{12}=v_{13}=1\}$ is a closed half-strip given by the inequalities $|x-y|\leq v_{14}v_{23}\leq x+y$ (see
\rf{half-strip} for the case of $v_{14}v_{23}=3$). 
Note that $\Pl_{4,2}(\Z)\cap\{v_{12}=v_{13}=1\}$ is the union of lattice points on two families of rays. The first family is given by
$$L_k=\{(x,y)\in\R^2 : y=x+v_{14}v_{23}-2k,\ x\geq k\},\quad\text{for }k=0,1,\dots, v_{14}v_{23}.$$
The corresponding lattice points are colored blue. 
The second family is 
$$M_k=\{(x,y)\in\R^2 : y=x+v_{14}v_{23}-2k+1,\ x\geq k\},\quad\text{for }k=1,\dots, v_{14}v_{23},$$
and its lattice points are colored green.

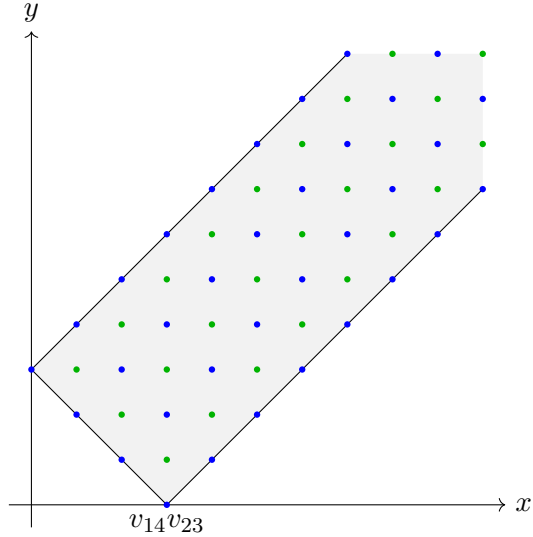
\begin{figure}[h]
\begin{tikzpicture}[scale=0.6]
  % plotting window (first quadrant)
  \def\xmax{10}
  \def\ymax{10}

  % dynamic intersection points
  \pgfmathsetmacro{\Bx}{\ymax - 3} % x where y = x+3 meets top (y=\ymax)
  \pgfmathsetmacro{\Dy}{\xmax - 3} % y where y = x-3 meets right (x=\xmax)

  % key vertices
  \coordinate (A) at (0,3);
  \coordinate (B) at (\Bx,\ymax);
  \coordinate (C) at (\xmax,\ymax);
  \coordinate (D) at (\xmax,\Dy);
  \coordinate (E) at (3,0);

  % shaded region
  \fill[gray!10] (A) -- (B) -- (C) -- (D) -- (E) -- cycle;

  % axes
  \draw[->] (-0.5,0) -- (\xmax+0.5,0) node[right] {$x$};
  \draw[->] (0,-0.5) -- (0,\ymax+0.5) node[above] {$y$};

  % boundary segments
  \draw[thin] (A) -- (B);
  \draw[thin] (E) -- (A);
  \draw[thin] (E) -- (D);

  % label next to point E
  \node[below] at (E) {$v_{14}v_{23}$};
  
  % lattice points inside the region (colored by odd/even diagonal)
  \foreach \x in {0,...,10}{
    \foreach \y in {0,...,10}{
      \pgfmathtruncatemacro{\sum}{\x+\y}
      \pgfmathtruncatemacro{\diff}{abs(\x-\y)}
      \pgfmathtruncatemacro{\c}{\y-\x}
      \ifnum\sum>2
        \ifnum\diff<4
          \ifodd\c
            \fill[blue] (\x,\y) circle (2pt);
          \else
            \fill[green!70!black] (\x,\y) circle (2pt);
          \fi
        \fi
      \fi
    }
  }
\end{tikzpicture}
\caption{Plücker set in $x=v_{24}, y=v_{34}$ coordinates}
\label{F:half-strip}
\end{figure}
Our goal is to construct $K_1,\dots, K_4\in\cP(\Z_2)$ which realize the points depicted in \rf{graph-xy}, where $(x,y)$ is a lattice point in $L_k$ or $M_k$.
    \begin{figure}[h]
    \begin{center}
    \includegraphics[scale=.48]{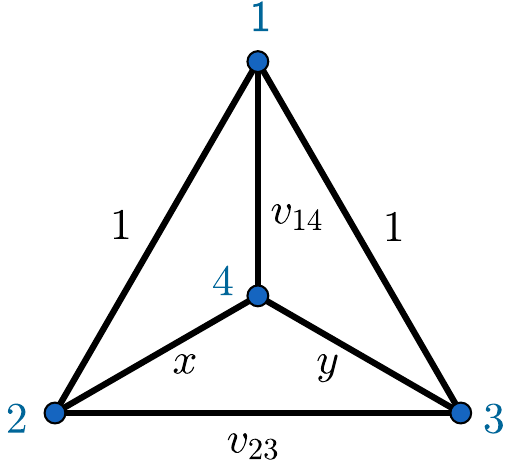} 
    \caption{A family of realizable points in $\Pl_{4,2}(\Z)$}
    \label{F:graph-xy} 
    \end{center}
    \end{figure}
Let the first three polygons be lattice segments
$$K_1=\conv\{0,e_1\},\quad K_2=\conv\{0,e_2\},\quad K_3=\conv\{0,v_{23}e_1+e_2\}.$$ 
Clearly, $\V(K_1,K_2)=\V(K_1,K_3)=1$, and $\V(K_2,K_3)=v_{23}$.
Now, to realize the lattice points in $L_k$ we choose $K_4$ to be the parallelogram
$$K_4=(x-k)K_1+\conv\{0,ke_1+v_{14}e_2\}.$$
Then $\V(K_1,K_4)=v_{14}$, and $\V(K_2,K_4)=x$. Also,
$$\V(K_3,K_4)=(x-k)\V(K_3,K_1)+
%\left|\det\left[
%\begin{matrix}
%v_{23} & k\\ 1 & v_{14}
%\end{matrix}
%\right]\right|
|\det(v_{23}e_1+e_2, ke_1+v_{14}e_2)|
=x-k+v_{14}v_{23}-k=y.$$
Next, to realize the lattice points in $M_k$ we choose $K_4$ to be the quadrilateral
$$K_4=\conv\left\{(0,0),\ (k-1,v_{14}),\ (x,v_{14}),\ (x-k, 0)\right\}.$$
Again, $\V(K_1,K_4)=\w_{e_2}(K_4)=v_{14}$, and $\V(K_2,K_4)=\w_{e_1}(K_4)=x$. 
Finally, let $u=\left(\begin{matrix}-1 \\ v_{23}\end{matrix}\right)$ be a primitive vector orthogonal to the segment $K_3$.
Then
$$\V(K_3,K_4)=\w_u(K_4)=\max_{w\in K_4}\langle w,u\rangle-\min_{w\in K_4}\langle w,u\rangle=(v_{14}v_{23}-k+1)-(k-x)=y.
$$

\end{pf}

\section{Non-realizable family in the interior of $\Pl_{4,2}(\Z)$}\label{S:interior-out}

In this final section we present an infinite family of interior points in $\Pl_{4,2}(\Z)$ which do not belong to $\SFV_{4,2}(\Z)$. 

\begin{Th}\label{T:non-realizable} Let $a,b\in\N$ with $d=\gcd(a,b)>1$. Then for any integers $c>1$ and $0<r<\min\left\{\frac{a}{d},\frac{b}{d},d\right\}$ the interior point $v$ of $\Pl_{4,2}(\Z)$ represented in \rf{graph-th}
does not belong to $\SFV_{4,2}(\Z)$.  
\end{Th}
\begin{figure}[h]
    \begin{center}
    \includegraphics[scale=.48]{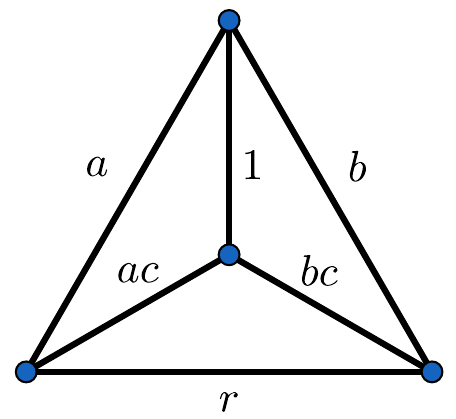} 
    \caption{A family of non-realizable interior points of $\Pl_{4,2}(\Z)$}
    \label{F:graph-th} 
    \end{center}
    \end{figure}

\begin{pf} First, it is easy to verify that $v$ is an interior point of $\Pl_{4,2}(\Z)$.
Suppose $v\in\SFV_{4,2}(\Z)$. Then there exist lattice polytopes $P_1,\dots,P_4$ such that
$v_{ij}=\V(P_i,P_j)$ and $v_{12}=1$, $v_{13}=a$, $v_{14}=b$, $v_{23}=ac$, $v_{24}=bc$, and $v_{34}=r$. 
Since $v_{12}=1$, by \rp{EG-2d}, either one of the $P_1,P_2$ is a primitive segment, or $P_1$ and $P_2$ are equal to the same unimodular triangle. The latter is impossible since $P_1=P_2$ would imply $v_{13}=v_{23}$, but $v_{13}<v_{23}$ by construction. If $P_1$ is a primitive segment, then $a=v_{13}=\w_{u}(P_3)$ and $b=v_{14}=\w_{u}(P_4)$, where $u$ is the primitive vector orthogonal to $P_1$. This contradicts part (1) of \rt{diagram}, as by construction, $r=v_{34}$ lies strictly between 0 and $\min\left\{d,\frac{\w_u(P_3)}{d},\frac{\w_u(P_4)}{d}\right\}$. Similarly, $P_2$ cannot be a primitive segment. 
\end{pf}

\begin{Rem} It follows from the above proof that the infinite family presented in \rt{non-realizable} is part of a larger subset of non-realizable interior points of $\Pl_{4,2}(\Z)$. Loosely speaking, as long as one coordinate of $v$, say $v_{12}$, equals 1 and the corresponding non-adjacent coordinate $v_{34}$ is small relative to $d=\gcd(v_{i2},v_{i3})$ and the quotients $v_{i2}/d$, $v_{i3}/d$ (for $i=1,2$)
then $v\not\in\SFV_{4,2}$. However, we do not have a complete description of the set of non-realizable interior points of $\Pl_{4,2}(\Z)$.
\end{Rem}

\bibliographystyle{alpha} % biblio style
\bibliography{lit.bib} 

\end{document}